\documentclass[12pt]{amsart}

\usepackage{amsmath,amstext,amssymb,amsopn,amsthm}
\usepackage{url,verbatim}
\usepackage{mathtools}
\usepackage{enumerate}

\usepackage{color,graphicx}

\usepackage[margin=30mm]{geometry}
\usepackage{eucal,mathrsfs,dsfont}%gives nicer set-names. 

\allowdisplaybreaks

\newtheorem*{theorem*}{Theorem}

\theoremstyle{definition}

\theoremstyle{remark}

\numberwithin{equation}{section}

\newcommand{\calG}{\mathcal{G}}
\newcommand{\calA}{\mathcal{A}}

\newcommand{\calX}{\mathcal{X}}

\newcommand{\R}{\mathds{R}}

\newcommand{\ol}{\overline}

\title[Billiard balls]{Protecting billiard balls from collisions}
\author{Jayadev Athreya and Krzysztof Burdzy}

\address{Department of Mathematics, Box 354350, University of Washington, Seattle, WA 98195}
\email{jathreya@uw.edu}
\email{burdzy@uw.edu}

\thanks{JSA's research was supported in part by NSF CAREER grant DMS 1559860. KB's research was supported in part by Simons Foundation Grant 506732. }

\begin{document}

\begin{abstract}

We present a game inspired by research on the possible number of billiard ball collisions in the whole Euclidean space. One player tries to place $n$ static ``balls'' with zero radius (i.e., points) in a way that will minimize the total number of possible collisions caused by the cue ball. The other player tries to find initial conditions for the cue ball to maximize the number of collisions. The value of the game is $\sqrt{n}$ (up to constants).
The lower bound is based on the Erd\H{o}s-Szekeres Theorem. The upper bound may be considered a generalization of the Erd\H{o}s-Szekeres Theorem.

\end{abstract}

\maketitle

\section{Introduction}\label{intro}

This paper is inspired by articles on the maximum number of totally elastic collisions for a finite system of $n$ balls in a billiard table with no walls (i.e, the whole Euclidean space). We will review the history of this problem in Section \ref{review}. 
It is a challenge to find initial conditions so that the ensuing evolution involves a large number of collisions. The first lower bound, given in \cite{BD}, was of order $n^3$ in dimensions $d\geq2$. This was later improved in \cite{BuI18} to an exponential lower bound in dimensions $d\geq 3$. One particularly simple set of initial conditions is to make $n-1$ balls static (i.e., their initial velocities are zero) and send the remaining ball (``cue ball'')  in a direction that would trigger a large number of collisions. Needless to say, these initial conditions are inspired by real billiards games.

We will use the above idea as an inspiration for a game---a simplified and idealized version of the original problem. 
Consider two players. The first player has to place $n$ identical balls at some locations in the Euclidean space. These balls are static. The other player has one cue ball. The goal of the second player is to give the initial position and velocity to the cue ball that will maximize the number of collisions (between the cue ball and other balls, and between the other balls). The goal of the first player is to place the balls in a way that will minimize the maximum number of collisions, with the maximum taken over all initial conditions of the cue ball.

If a moving ball strikes a static ball then they will move in directions that form the right angle (assuming that the balls have the same masses and radii). This rigid property of elastic collisions suggests that the first player should place the balls at very large distances because then it will be hard for the second player to arrange for each of the two balls involved in a collision to move in a direction that will result in a new collision with some other ball. 

The above remarks lead us to the following idealized model. The static balls are assumed to be far from each other, in the sense that the ratio of the ball diameter to the typical distance between balls is very small. We will take this idea to the limit and we will assume that the balls' radii are zero. In other words, the cue ball, also assumed to be pointwise, will have to  hit a point, not a ball. Second, only one of the two balls (points) involved in a collision will be allowed to be involved in another collision, i.e., only one billiard trajectory may emanate from a point of the collision. The unique billiard trajectory leaving a collision point will have to form an angle greater than $\pi/2$ with the trajectory of the ball arriving at the collision location. Finally, moving balls will be able to collide only with static balls (points) because it is ``unlikely'' that two moving balls can collide. See Fig.~\ref{fig1}.

\begin{figure} \includegraphics[width=0.5\linewidth]{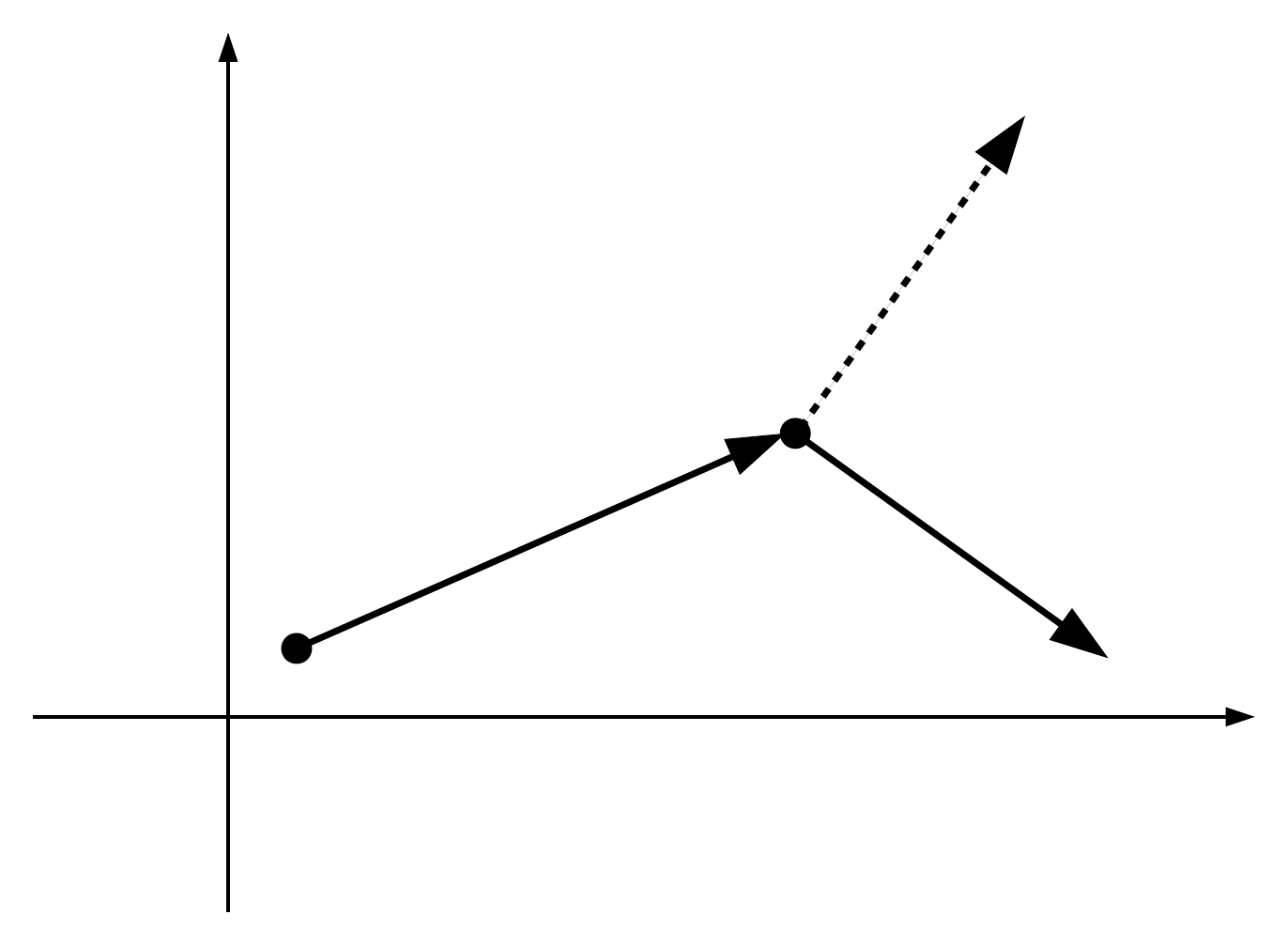}
\caption{ Only one billiard trajectory may emanate from a point of the collision. The dashed line represents the trajectory of the ball that will not be involved in any future collisions.
}
\label{fig1}
\end{figure}

Informally speaking, our main result says that, in two dimensions, the value of the game is $c\sqrt{n}$ (in the rigorous statement, the bounds for both players have unequal constants in front of $\sqrt{n}$). The first player may place the balls (points) in such a way that the second player can create at most $3(\sqrt{n}+1)$ collisions, and no matter where the balls are placed by the first player, the second one can generate at least $\lfloor\sqrt{n-1} +1\rfloor$ collisions.

A rigorous mathematical representation of the game and our claims is given in Section \ref{model}. We  analyze only the two dimensional case.

Our proof of the lower bound $\lfloor\sqrt{n-1} +1\rfloor$ is based on the Erd\H{o}s-Szekeres Theorem. The upper bound $3(\sqrt{n}+1)$ may be considered a generalization of (one direction of) the Erd\H{o}s-Szekeres Theorem.
The Erd\H{o}s-Szekeres Theorem was proved in \cite{ESz}. For short proofs, see \cite{Black,Seid}.

There are  many possible versions and generalizations of the game. One of the most obvious modifications is to allow the angle between consecutive portions of the trajectory to be in the interval $[\alpha, \pi]$ for some $\alpha \ne \pi/2$.
We do not know any theorems about this version of the game, except for the obvious monotonicity; for example, if $\alpha < \pi/2$ then the second player can generate at least as many collisions as in the original game (when the lower bound for the angle is equal to $\pi/2$).

\subsection{Hard ball collisions---historical review}
\label{review}

The question of whether a finite system of hard balls can have an infinite number of elastic collisions was posed by Ya.~Sinai. It was answered in the negative by \cite{Vaser79}. For alternative proofs see \cite{Illner89, Illner90,IllnerChen}. 
It was proved in \cite{BFK1} that a system of $n$ balls in the Euclidean space undergoing elastic collisions can experience at most
\begin{align}\label{s26.1}
\left( 32 \sqrt{\frac{m_{\text{max}}}{m_{\text{min}}} } 
\frac{r_{\text{max}}}{r_{\text{min}}} n^{3/2}\right)^{n^2}
\end{align}
collisions. Here $m_{\text{max}}$ and $m_{\text{min}}$ denote the maximum and the minimum masses of the balls. Likewise, $r_{\text{max}}$ and $r_{\text{min}}$ denote the maximum and the minimum radii of the balls.
The following alternative upper bound for the maximum number of collisions appeared in \cite{BFK5}
\begin{align}\label{s26.2}
\left( 400 \frac{m_{\text{max}}}{m_{\text{min}}} 
 n^2\right)^{2n^4}.
\end{align}
  The papers \cite{BFK1,BFK2, BFK3,BFK4, BFK5} were the first to present universal bounds \eqref{s26.1}-\eqref{s26.2} on the number of collisions of $n$ hard balls in any dimension. No improved universal bounds were found since then, as far as we know.

It has been proved in \cite{BD} by example that the  number of elastic collisions of $n$ balls
in $d$-dimensional space is greater than $n^3/27$ for $n\geq 3$ and $d\geq 2$, for some initial conditions. The previously known lower bound was of order $n^2$ (that bound was for balls in dimension 1 and was totally elementary). An exponential lower bound was given in \cite{BuI18} in dimensions $d\geq 3$.

A related article, \cite{ABD}, gives an upper bound for the number of ``collisions'' of pinned billiard balls.

\section{Rigorous model}
\label{model}

Whenever we refer to the angle between two line segments sharing an endpoint, we mean the smaller of the two angles, i.e., the one in the range $[0,\pi]$.

Let $\calA(n)$ be the family of all sets of $n$ distinct points in $\R^2$.
Given a set $\calX:=\{x_1, x_2, \dots, x_n\}\in \calA(n)$, we define an \emph{admissible billiard trajectory} $\Gamma$ as a polygonal line with vertices $y_1, y_2, \dots, y_k$, such that $k\leq n$, all   $y_1, y_2, \dots, y_k$ are distinct, each $y_j$ is equal to some $x_m$, and the  angle between the line segments $\ol {y_{j-1}, y_j}$ and  $\ol {y_{j}, y_{j+1}}$ is in the range $(\pi/2, \pi]$ for all $2\leq j \leq k-1$. We will call $k$ the length of $\Gamma$ and we will write $|\Gamma| =k$. The set of all admissible billiard trajectories relative to $\calX\in \calA(n)$ will be denoted $\calG(\calX)$. See Fig.~\ref{fig2}.

\begin{figure} \includegraphics[width=0.5\linewidth]{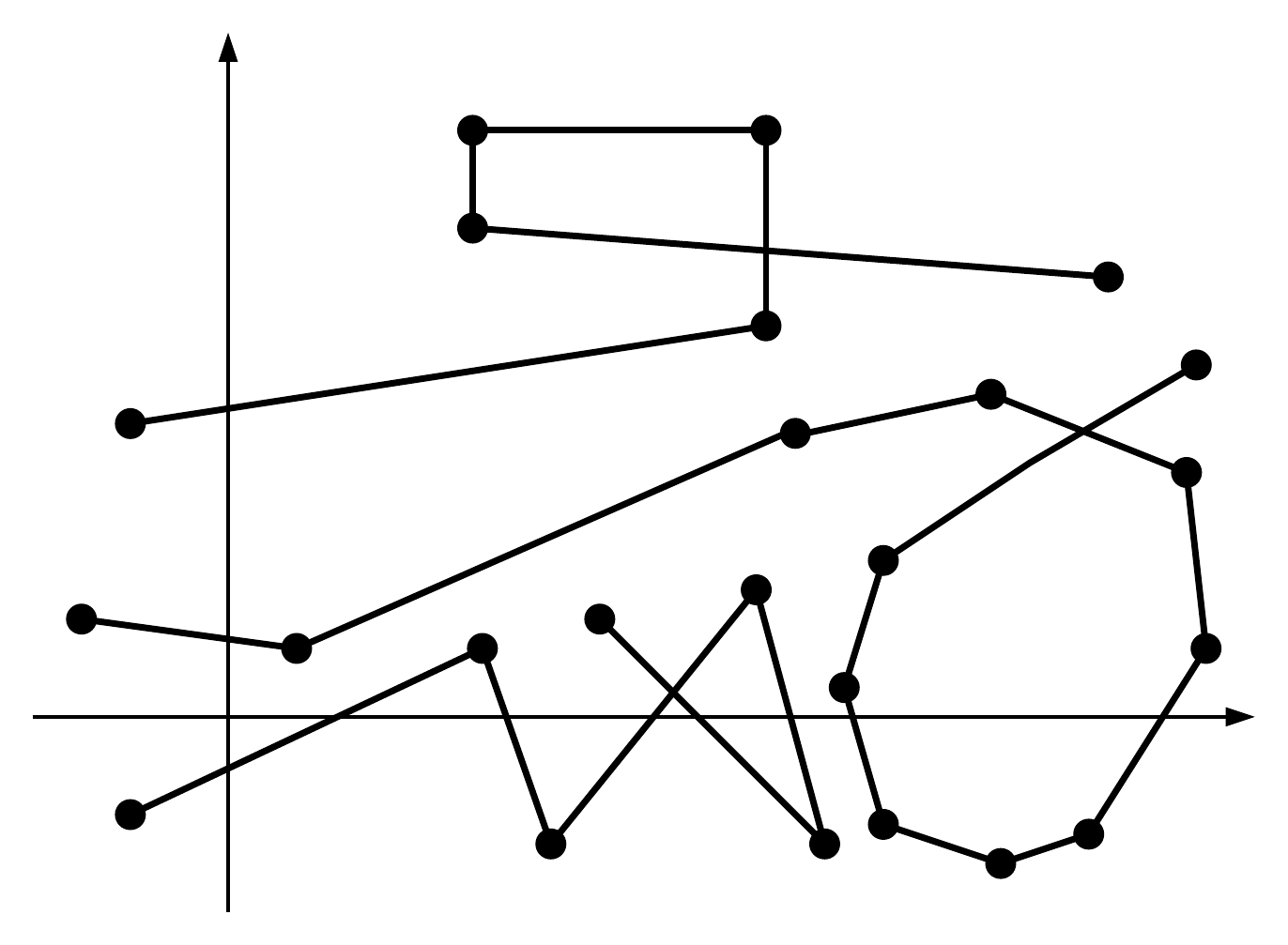}
\caption{The top trajectory is admissible. The lower trajectory  is not admissible.
}
\label{fig2}
\end{figure}

\begin{theorem*}
The following holds for every $n\geq 1$.

(i) For every $\calX\in \calA(n)$ there exists $\Gamma \in \calG(\calX)$ with $|\Gamma| \geq \lfloor\sqrt{n-1} +1\rfloor$.

(ii) There exists $\calX\in \calA(n)$ such that $|\Gamma| \leq  3 \Big\lceil \sqrt{n} \Big\rceil$ for every $\Gamma \in \calG(\calX)$.

\end{theorem*}

\begin{proof}

(i) Let the coordinates of $x_k$'s be denoted $x_k=(x_k^1, x_k^2)$. We can assume without loss of generality that the set $\calX$ is oriented so that all coordinates $x_1^1, x_2^1, \dots , x_n^1$ are distinct, and the same is true for $x_1^2, x_2^2, \dots , x_n^2$.

Let $m =\lfloor\sqrt{n-1} +1\rfloor$.
By the Erd\H{o}s-Szekeres Theorem, there exists a sequence $k_1 < k_2< \dots< k_m$  such that the sequence $x^2_{k_1}, x^2_{k_2}, \dots , x^2_{k_m}$ is monotone. This implies that 
 the  angle between the line segments $\ol {x_{k_{j-1}}, x_{k_j}}$ and  $\ol {x_{k_{j}}, x_{k_{j+1}}}$ is in the range $(\pi/2, \pi]$ for all $2\leq j \leq m-1$.

\bigskip

(ii) 
Let $m = \Big\lceil \sqrt{n} \Big\rceil$. For $r>0$, let $K(r) = K(m,r)$ be the set of vertices of a regular $m$-gon inscribed in a circle of radius $r$. That is, $K(r)$ consists of the points (in complex notation) $r e ^{(2k\pi/m)i}$, for $k=0,1,\dots, m-1$.  See Fig.~\ref{fig5}.

\begin{figure} \includegraphics[width=0.3\linewidth]{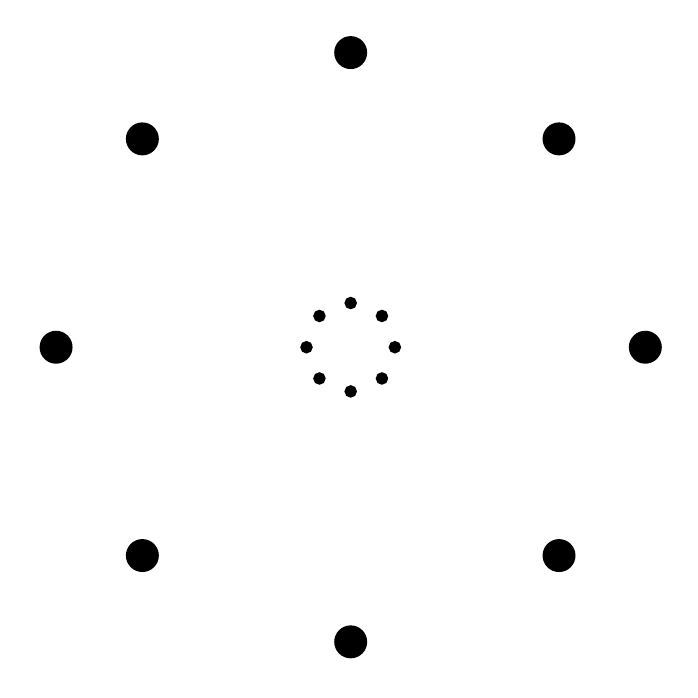}
\caption{Sets $K(8,r_1)$ and $K(8,r_2)$.
}
\label{fig5}
\end{figure}

We can find a scaling factor $a=a(m)\in(0,1)$ sufficiently small so that if $x \in K(ar) $ and $y,z\in K(r)$ then the angle between line segments $\ol{xy}$ and $\ol{yz}$ is strictly less than $\pi/2$. For such an $a$, a stronger property holds, namely,  

(*)
 If $b\in (0,1]$, $c\in(0,a]$, $x \in K(cr) $, $y\in K(r)$ and $z\in K(br)$ then the angle between line segments $\ol{xy}$ and $\ol{yz}$ is strictly less than $\pi/2$.

Let $\calX = \bigcup_{j=0}^{m-1} K(a^j )$. 
We have $|\calX|=m^2=\Big\lceil \sqrt{n} \Big\rceil^2\geq n$.
The validity of the argument given below is not affected by the possibility that $\calX$ may have more than $n$ elements. 

Suppose that $\Gamma\in \calG(\calX)$.
Let $k$ be the largest number with the property that 
some of the vertices of $\Gamma$ belong to $K(a^k)$.
The number of vertices of $\Gamma$ in $K(a^k)$ is at most $m$, i.e., the cardinality of $K(a^k)$. 

 Let the line segments in $\Gamma$ be denoted 
$\ol{y_j, y_{j+1}}$ for $j=1,\dots, \ell-1$, where $\ell = |\Gamma|$.
Let $I:=(p,p+1, \dots, r)$ be 
a maximal interval of consecutive integers such that $y_i \in K(a^k)$ for all $i\in I$.
Here ``maximal'' means that either $p= 1$ or $y_{p-1} \notin K(a^k)$, and, similarly, either $r= \ell$ or $y_{r+1} \notin K(a^k)$. 
We have $r-p +1\leq m$ because  the number of vertices of $\Gamma$ in $K(a^k)$ is at most $m$.

Suppose that $r < \ell$. Then $y_{r+1} \in K(a^{k_1})$ for some $k_1 < k$. 
It follows from (*) that
if $r+1 < \ell$ 
then $y_{r+2} \in K(a^{k_2})$ for some $k_2 < k_1$.
By induction, if $r+j < \ell$ 
then $y_{r+j+1} \in K(a^{k_{j+1}})$ for some $k_{j+1} < k_j$.
Since $k\leq m-1$ and $k>k_1 > \dots > k_j \geq 0$ for all $j$, it follows that $\ell - r \leq m-1$.

A completely analogous argument shows that $p \leq m-1$. Hence, 
\begin{align*}
|\Gamma|=\ell = (p-1) + (r-p+1) + (\ell -r) \leq (m-1) + m + (m-1) < 3 m = 3 \Big\lceil \sqrt{n} \Big\rceil.
\end{align*}

\end{proof}

\medskip

\noindent\textbf{Remark.} The space $\calG(\calX)$ is an interesting configuration space. It seems interesting to try and understand the topology of $\calG(\calX)$. The topology of $\calA(n)$ is a well-studied subject, as it is the \emph{classifying space} of the braid group $B_n$, see~\cite[pp.183-186]{Arn}.

\bibliographystyle{plain}
\bibliography{hard}

\end{document}